\documentclass[10pt]{article}

\usepackage{amsmath,amsfonts}
\usepackage{algorithmic}
\usepackage{algorithm}
\usepackage{array}
\usepackage[caption=false,font=normalsize,labelfont=sf,textfont=sf]{subfig}
\usepackage{textcomp}
\usepackage{stfloats}
\usepackage{url}
\usepackage{verbatim}
\usepackage{graphicx}
\usepackage{cite}
\usepackage{xcolor}
\usepackage{amsthm}
\newtheorem{theorem}{Theorem}
\newtheorem{lemma}{Lemma}
\newtheorem{assumption}{Assumption}
\newtheorem{corollary}{Corollary}

\begin{document}

\title{Structure-preserving non-linear PCA for matrices}

\author{Joni Virta, Andreas Artemiou\thanks{Joni Virta is with the Department of Mathematics and Statistics, University of Turku, Finland (email: joni.virta@utu.fi). Andreas Artemiou is with the Department of Information Technologies, University of Limassol, Cyprus (email: artemiou@uol.ac.cy).}
}



\maketitle

\begin{abstract}
We propose MNPCA, a novel non-linear generalization of (2D)$^2${PCA}, a classical linear method for the simultaneous dimension reduction of both rows and columns of a set of matrix-valued data. MNPCA is based on optimizing over separate non-linear mappings on the left and right singular spaces of the observations, essentially amounting to the decoupling of the two sides of the matrices. We develop a comprehensive theoretical framework for MNPCA by viewing it as an eigenproblem in reproducing kernel Hilbert spaces. We study the resulting estimators on both population and sample levels, deriving their convergence rates and formulating a coordinate representation to allow the method to be used in practice. Simulations and a real data example demonstrate MNPCA's good performance over its competitors.
\end{abstract}


\section{Introduction}

The diverse forms of data encountered in modern applications have caused a surge in the development of statistical methods specializing to datasets that do not exhibit the standard form of samples of points in $\mathbb{R}^p$. In this work, we focus on one of these special cases, matrix-valued data, where a sample $X_1, \ldots , X_n$ of $p_1 \times p_2$ matrices is observed. In typical applications, such as imaging, the dimensions $p_1, p_2$ can be very large in size and the first step of the analysis is often dimension reduction.

Perhaps the most well-known statistical dimension reduction technique for matrix-valued data is a generalization of the classical PCA known as (2D)$^2${PCA} \cite{zhang20052d}, where the observed matrices are replaced with $d_1 \times d_2$ latent matrices $Z_i := A' (X_i - \bar{X} ) B$ where $\bar{X}$ is the sample mean matrix and $A, B$ contain, respectively, any first $d_1$ and $d_2$ eigenvectors of the matrices
\begin{align*}
    \frac{1}{n}\sum_{i=1}^n X_i X_i' - \bar{X} \bar{X}' \quad \mbox{and} \quad \frac{1}{n}\sum_{i=1}^n X_i' X_i - \bar{X}' \bar{X}.
\end{align*}
This reduction can be seen to be natural in the following two senses. (a) It replaces a sample of matrices with a sample of matrices, essentially \textit{preserving the type of the data}. This is not the case for many dimension reduction methods (such as the kernel methodology listed below) which instead produces a vector of scores whose relation to the original matrix structure is not clear. (b) The latent matrices $Z_i$ exhibit a row-column dependency structure where two elements of $Z_i$ that share a row (column) also share the same projecting column of $A$ ($B$). These properties let (2D)$^2${PCA} properly leverage the structure of the observed matrices and see them as more than simply collections of elements. (2D)$^2${PCA} has been used to great success in various applications, such as face recognition \cite{zhang20052d} and stock price prediction \cite{gao2016deep}.

In this work we propose MNPCA (short for ``matrix non-linear PCA''), a non-linear extension of (2D)$^2${PCA} that retains both of the properties listed in the previous paragraph. That is, we construct a non-linear mapping $X_i \mapsto g(X_i) =: Z_i$ such that (a) $Z_1, \ldots , Z_n$ are $d_1 \times d_2$ matrices, (b) $g$ imposes specific dependencies between the rows and between the columns of the images $Z_i$ (i.e., two elements sharing a row are more similar than elements on different rows), and (c) when a linear kernel is used, the mapping $g$ reduces to the linear (2D)$^2${PCA}.  Analogously to (2D)$^2${PCA}, our non-linear mapping $g$ can thus be seen to preserve the matrix structure of the original data. This behavior of $g$ is in strict contrast to existing methods of non-linear dimension reduction of matrix-valued data which we review next. All of the methods listed below are based either on (2D)$^2${PCA} \cite{zhang20052d} or its one-sided precursor 2DPCA \cite{yang2004two} which applies only the transformation $A$ or $B$ but not both.

The authors in \cite{kong2005generalized} defined kernel 2DPCA (K2DPCA), which is essentially equivalent to applying the standard kernel PCA to the set of all $n p_1$ rows in the matrix sample. Independently, \cite{nhat2007kernel} proposed a method equivalent to K2DPCA of \cite{kong2005generalized}. In \cite{zhang2006recognizing} K2DPCA was applied separately to the rows and columns of the input matrices and standard PCA was used on the resulting pairs of latent representations to obtain combined latent variables. Whereas, \cite{yu2009k2dpca} first used K2DPCA to reduce the number of rows in the data and then applied regular 2DPCA to the obtained latent matrices to reduce also their column dimension. To summarize, the literature on non-linear extensions of two-dimensional PCA either focuses on one-sided methods (extensions of 2DPCA) or combines a pair of one-sided methods into a two-sided method (i.e., one that reduces both rows and columns simultaneously) in a theoretically cumbersome way. Both options can be seen as sub-optimal: If the number of columns in the data is even moderately large, reducing only the number of rows still leaves the data dimension high, making, e.g., the visualization of the resulting components impossible. Whereas, in chaining the row and column reductions, the outcome either depends non-trivially on the order in which the rows and columns are reduced \cite{yu2009k2dpca} or is artificial and loses the structural connection to the original matrices \cite{zhang2006recognizing}.




A further problem underlying the methods listed above is their high computational complexity. Namely, as K2DPCA operates on the sample of all $np_1$ rows of the data, the size $np_1 \times np_1$ of the corresponding kernel matrix can be enormous even for combinations of a moderate sample size $n$ and number of rows $p_1$. To combat this, \cite{zhang2006recognizing} propose approximating the full kernel matrix with the kernel matrix of the within-observation row means of the data, but this leads to the loss of the row structure and it is not clear how it affects the accuracy of the method. 


Our proposed method for constructing the non-linear mapping $X_i \mapsto g(X_i)$ avoids the previous pitfalls by working not with the observed matrices $X_i$ themselves but with their singular value decompositions (SVD), $X_i = U_i D_i V_i'$. This simple change of perspective has two major implications:
\begin{enumerate}
    \item The singular value decomposition essentially ``decouples'' the row and column information in the input matrices, allowing the independent and simultaneous reduction of the row and column spaces, possibly with different kernel functions. As a consequence, MNPCA is fully two-sided and the order in which the two sides are reduced does not affect the outcome.
    \item As leading singular vectors capture the main directions of data variation, truncating the SVD allows us to leverage (almost) the full data information while simultaneously avoiding the inflation of the size of the kernel matrix.
\end{enumerate}

The primary contributions of this work are: (i) We formulate the population-level version of MNPCA, our proposed SVD-based non-linear extension of (2D)$^2${PCA}. As is typical in the literature on non-linear dimension reduction \cite{fukumizu2004dimensionality}, this requires casting the problem into the framework of reproducing kernel Hilbert spaces (RKHS) and Hilbert-Schmidt (H-S) operators. Special attention is paid to formulating the exact assumptions under which MNPCA is well-defined. (ii) We study the asymptotic properties of MNPCA and derive the convergence rate of the corresponding estimator. (iii) We derive a coordinate representation for MNPCA, allowing its sample-level implementation, and discuss the selection of its tuning parameters. (iv) We compare MNPCA to several of its competitors using both simulations and a real data example. Note that earlier work on the non-linear extensions of (2D)$^2${PCA} (see the list of references earlier) has focused solely on points (iii) and (iv), ignoring the finer theoretical aspects of the corresponding methods.


Finally, we briefly note that the structure-acknowledging non-linear dimension reduction of matrices can be approached also from another viewpoint besides (2D)$^2${PCA}. Namely, \cite{liu2013tensorial} apply standard kernel PCA to the sample $X_1, \ldots, X_n$ but with a very specific choice of kernel that recognizes the matrix structure of the data. While interesting, this approach goes somewhat against the spirit of kernel methodology where the kernel function is typically seen as a tuning parameter and its choice is equivalent to determining what kind of latent structures one is after. The limiting to a single kernel in \cite{liu2013tensorial} is in strict contrast to our proposed method which allows the use of any kernel function that is either odd or even, see the definition in Section \ref{sec:notation}. This restriction is a mild one as the classes of odd and even kernels are very large (in fact, every positive semi-definite kernel induces both a corresponding odd and an even kernel). 

The manuscript is organized as follows. Section \ref{sec:notation} begins with some definitions and notation. In Section \ref{sec:projections} we use the combination of SVD and even/odd kernels to motivate a well-defined non-linear analogy for the two-sided projection of a matrix. In Section \ref{sec:population} we formulate the population version of the MNPCA-procedure. Section \ref{sec:asymptotics} focuses on the asymptotic properties of the sample version of the method, whereas its coordinate representation is constructed in Section \ref{sec:coordinate}. Tuning parameter selection is discussed in Section \ref{sec:tuning}. Simulations and real data examples are given in Section \ref{sec:examples} and Section \ref{sec:discussion} concludes with a discussion. All proofs of technical results are collected to Appendix.

\section{Notation and definitions}\label{sec:notation}


As is standard, we use the word kernel to refer to any $\kappa: \mathbb{R}^{p} \times \mathbb{R}^p \to \mathbb{R}$ that is symmetric and positive semi-definite \cite{shawe2004kernel}. Let next $\kappa_1, \kappa_2$ be continuous kernels. We denote the RKHS induced by $\kappa_1$ and $\kappa_2$ by $(\mathcal{H}_1, \langle \cdot , \cdot \rangle_1 ), (\mathcal{H}_2, \langle \cdot , \cdot \rangle_2)$, respectively. We use the notation $\mathcal{B}(\mathcal{H}_i, \mathcal{H}_j)$ to refer to the set of all bounded linear operators from $\mathcal{H}_i$ to $\mathcal{H}_j$. By the continuity of the kernels $\kappa_1, \kappa_2$, the spaces $\mathcal{H}_1, \mathcal{H}_2$ are separable and admit orthonormal bases \cite{hein2004kernels}.

Recall that an operator $\mathcal{F} \in \mathcal{B}(\mathcal{H}_i, \mathcal{H}_j)$ is said to be a Hilbert-Schmidt operator if the quantity $\sum_{k=1}^\infty \sum_{\ell=1}^\infty \langle v_\ell, F u_k \rangle_j^2$ is finite for some orthonormal bases $\{ u_k \} $ and $ \{ v_\ell \} $ of $\mathcal{H}_i$ and $\mathcal{H}_j$, respectively, in which case its value does not depend on the choice of the bases and is termed the squared Hilbert-Schmidt norm $\| F \|^2_{ij}$ of $F$. The subscript $ij$ in the norm signifies that the domain of the norm consists of linear operators from $\mathcal{H}_i$ to $\mathcal{H}_j$. The vector space $\mathcal{S}(\mathcal{H}_i, \mathcal{H}_j)$ of all Hilbert-Schmidt operators in $\mathcal{B}(\mathcal{H}_i, \mathcal{H}_j)$ itself becomes a Hilbert space when endowed with the inner product
\begin{align*}
    \langle F_1, F_2 \rangle_{ij} := \sum_{k=1}^\infty \sum_{\ell=1}^\infty \langle v_\ell, F_1 u_k \rangle_j \langle v_\ell, F_2 u_k \rangle_j.
\end{align*}

For $f \in \mathcal{H}_i$ and $g \in \mathcal{H}_j$, the tensor product $g \otimes f$ denotes the element of $\mathcal{S}(\mathcal{H}_i, \mathcal{H}_j)$ acting as $(g \otimes f)h = \langle f, h \rangle_i g$ for any $h \in \mathcal{H}_i$. It is straightforwardly checked that $ \| g \otimes f \|_{ij} = \| f \|_i \| g \|_j $. Finally, we note that the Hilbert-Schmidt norms satisfy $ \langle F_1, F_2 \rangle_{ij} = \langle F_1^*, F_2^* \rangle_{ji} $ and $\langle g, F_1 f \rangle_j = \langle (g \otimes f), F_1 \rangle_{ij}$ for all $F_1, F_2 \in \mathcal{S}(\mathcal{H}_i, \mathcal{H}_j)$, $f \in \mathcal{H}_i$ and $g \in \mathcal{H}_j$.

A key role in our development is played by the so-called even and odd kernels, detailed in the following lemma, given originally as Corollary 1 in \cite{krejnik2012reproducing}.

\begin{lemma}\label{lem:odd_even_kernels}
Let $\kappa: \mathbb{R}^{p} \times \mathbb{R}^p \to \mathbb{R}$ be a kernel that satisfies $\kappa(x, y) = \kappa(-x, -y)$ for all $x, y \in \mathbb{R}^{p}$. Then,
\begin{itemize}
    \item[a)] The function $\kappa^-: \mathbb{R}^p \times \mathbb{R}^p \to \mathbb{R}$ acting as $(x, y) \mapsto \kappa(x, y) - \kappa(-x, y)$ is a kernel in $\mathbb{R}^p$ which is \textbf{odd} in the sense that $\kappa^-(-x, y) = - \kappa^-(x, y)$ for all $x, y \in \mathbb{R}^{p}$.
    \item[b)] The function $\kappa^+: \mathbb{R}^p \times \mathbb{R}^p \to \mathbb{R}$ acting as $(x, y) \mapsto \kappa(x, y) + \kappa(-x, y)$ is a kernel in $\mathbb{R}^p$ which is \textbf{even} in the sense that $\kappa^+(-x, y) = \kappa^+(x, y)$ for all $x, y \in \mathbb{R}^{p}$.
\end{itemize}
\end{lemma}

Lemma \ref{lem:odd_even_kernels} essentially states that, given any kernel $\kappa$, one can always construct the corresponding odd and even kernels $\kappa^-, \kappa^+$. In the sequel , we say that $\kappa^-, \kappa^+$ are the odd and even kernels induced by $\kappa$.

\section{Two-sided non-linear projections}\label{sec:projections}

Let $X$ be a random $p_1 \times p_2$ matrix. In linear dimension reduction for matrix data, the objective is to search for directions $a \in \mathbb{S}^{p_1 - 1}$, $b \in \mathbb{S}^{p_2 - 1}$, where $\mathbb{S}^{p - 1}$ denotes the unit sphere in $\mathbb{R}^p$, such that the two-sided projection $a' X b$ provides a meaningful reduction to the data. Having obtained the directions $a_1, \ldots , a_{d_1}$ and $b_1, \ldots , b_{d_2}$ (typically under orthogonality constraints within the two sets), their combinations yield a total of $d_1 d_2$ projections that are most conveniently arranged into the $d_1 \times d_2$ matrix $Z := (a_j' X b_k)_{j = 1}^{d_1}{}_{k = 1}^{d_2}$. This dimension reduction can be seen to preserve the matrix structure of the data as, indeed, each row of $Z$ shares the same $a$-vector and each column of $Z$ the same $b$-vector. Methods subscribing to this paradigm include, e.g., \cite{zhang20052d, li2010dimension, xue2014sufficient, ding2015tensor, virta2017independent, virta2018jade}.


In this section we formulate a \textit{non-linear} extension of this concept (simultaneous two-sided projection) using RKHS. Our objective with the extension is to preserve the previous idea that extracting a total of $d_1$ ``left'' directions and $d_2$ ``right'' directions gives us a $d_1 \times d_2$ reduced matrix where the rows share the same ``non-linear direction'' and similarly for columns. In the sequel, we let $\kappa_1: \mathbb{R}^{p_1} \times \mathbb{R}^{p_1} \to \mathbb{R}$ and $\kappa_2: \mathbb{R}^{p_2} \times \mathbb{R}^{p_2} \to \mathbb{R}$ denote the kernels corresponding to the two sides and we make the following assumption regarding them.

\begin{assumption}\label{assu:kernel_assumptions}
The kernels $\kappa_1, \kappa_2$ are either both odd or both even.
\end{assumption}

The oddness/evenness of the two kernels is required later on to ensure that the lack of fixed signs in singular value decomposition does not compromise the uniqueness of our projections. We additionally make the following assumption regarding the data $X$.

\begin{assumption}\label{assu:singular_values_simple}
For some $r \leq \min \{ p_1, p_2 \} $, the random matrix $X$ has almost surely rank $r$ and its non-zero singular values are almost surely simple.
\end{assumption}

The first part of Assumption \ref{assu:singular_values_simple} (almost surely fixed rank) is made for convenience and could easily be omitted at the cost of more cluttered notation. The second part (almost surely simple singular values) is satisfied, in particular, if $X$ has an absolutely continuous distribution w.r.t. the Lebesgue measure (in which case the rank $r = \min \{ p_1, p_2 \} $ almost surely).

Let $(u_j, v_j) \equiv (u_j(X), v_j(X))$, $j = 1, \ldots , r$, denote any pair of $j$th left and right singular vectors of the random matrix $X$. Under Assumption \ref{assu:singular_values_simple}, each of the pairs $(u_j, v_j)$, $j = 1, \ldots , r$, is almost surely uniquely defined up to the joint sign of the members of the pair. That is, if $(u_j, v_j)$ is a $j$th singular pair of $X$, then the only other $j$th singular pair of $X$ is $(-u_j, -v_j)$. We denote the $j$th singular value of $X$ by $\sigma_j \equiv \sigma_j(X)$.

Let now $f \in \mathcal{H}_1$ and $g \in \mathcal{H}_2$ be arbitrary functions that play the role of the projection directions $a \in \mathbb{R}^{p_1}, b \in \mathbb{R}^{p_2}$ in our non-linear extension. We define the two-sided projection of $X$ to the pair $(f, g)$ to be
\begin{align}\label{eq:reduced_variable}
\begin{split}
    & \sum_{j = 1}^r \sigma_j f ( u_j ) g ( v_j ) \\
    =&  \left\langle f, \left( \sum_{j = 1}^r \sigma_j \{ \kappa_1( \cdot,  u_j ) \otimes \kappa_2( \cdot,  v_j ) \} \right) g \right\rangle_1 \\
    =:& \langle f, U g \rangle_1,
\end{split}
\end{align}
where the random operator $U \equiv U(X) := \sum_{j = 1}^r \sigma_j \{ \kappa_1( \cdot,  u_j ) \otimes \kappa_2( \cdot,  v_j ) \}$ can be seen as the non-linear ``representation'' of the random matrix $X$. That $U$ is a Hilbert-Schmidt operator follows straightforwardly, see \eqref{eq:operator_norm_bound} in Section~\ref{sec:population}. The projection \eqref{eq:reduced_variable} is an exact non-linear analogy of the linear projection $ a' X b = \sum_{j = 1}^r \sigma_j a'u_j  b'v_j$, to which it reduces when the kernels $\kappa_1, \kappa_2$ are linear. The oddness or evenness of the kernels $\kappa_1, \kappa_2$ guarantees that the joint sign of any individual singular pair plays no role in the construction of the projection. E.g., if both kernels are odd, we have for all $u \in \mathbb{R}^{p_1}$, $v \in \mathbb{R}^{p_2}$ that
\begin{align*}
     \kappa_1( \cdot,  -u ) \otimes \kappa_2( \cdot,  -v )  =& \{-\kappa_1( \cdot,  u ) \} \otimes \{ - \kappa_2( \cdot,  v )\}\\
     =& \kappa_1( \cdot,  u ) \otimes \kappa_2( \cdot,  v ).
\end{align*}
Consequently, the reduced variable $\langle f, U g \rangle_1$ in \eqref{eq:reduced_variable} is almost surely uniquely defined, regardless of which particular singular value decomposition of $X$ we use. Note that this would not be the case without the second part of Assumption \ref{assu:singular_values_simple} as then more freedom would be allowed in choosing the singular vectors corresponding to singular values with multiplicity greater than one, and this freedom would not, in general, be cancelled out unless we used linear kernels.

\section{MNPCA}\label{sec:population}

Recall that (2D)$^2${PCA} \cite{zhang20052d} is a method of linear dimension reduction that can be seen as an extension of principal component analysis to matrices (when $p_2 = 1$ it is equivalent to the standard PCA). In (2D)$^2${PCA}, the left projection directions $a_1, \ldots, a_{d_1}$ are found as the first $d_1$ eigenvectors of the matrix $\mathrm{E} [ \{ X - \mathrm{E}(X) \} \{ X - \mathrm{E}(X) \}' ]$, whereas their right-hand side counterparts $b_1, \ldots, b_{d_2}$ are analogously taken to be the first $d_2$ eigenvectors of the matrix $\mathrm{E} [ \{ X - \mathrm{E}(X) \}' \{ X - \mathrm{E}(X) \} ]$. Given the projection directions, the reduced matrix $Z$ containing the $d_1d_2$ combined projections is formed as $Z := (a_k' \{ X - \mathrm{E}(X) \} b_\ell)_{k = 1}^{d_1}{}_{\ell = 1}^{d_2}$. If multiple eigenvalues are encountered, the corresponding eigenvectors and projections are not uniquely defined. 

Prior to defining our non-linear analogue of (2D)$^2${PCA}, we first construct the first and second moments of the operator $U$ defined in Section \ref{sec:projections}. For this, we make a weak assumption about the kernels $\kappa_1, \kappa_2$ that will simplify the presentation to come. We note that this assumption is made simply for convenience and that our theory would function perfectly well even without it, assuming that the later moment assumptions are suitably adjusted. 

\begin{assumption}\label{assu:bounded_on_unit_sphere}
There exists constants $C_1, C_2 > 0$ such that, for all $u \in \mathbb{S}^{p_1 - 1}, v \in \mathbb{S}^{p_2 - 1}$,
\begin{align*}
    \kappa_1(u, u) < C_1 \quad \mbox{and} \quad \kappa_2(v, v) < C_2.
\end{align*}
\end{assumption}

Assumption \ref{assu:bounded_on_unit_sphere} is satisfied, e.g., for even and odd kernels induced (in the sense of Lemma \ref{lem:odd_even_kernels}) by all Gaussian, Laplace and polynomial kernels.


Recall then that we defined the random operator $ U $ in Section \ref{sec:projections} as
\begin{align*}
    U = \sum_{j = 1}^r \sigma_j \{ \kappa_1( \cdot,  u_j ) \otimes \kappa_2( \cdot,  v_j ) \},
\end{align*}
where $(u_j, v_j)$ is a $j$th singular pair of the almost surely rank-$r$ random matrix $X$ and $\sigma_j$ denotes the corresponding singular value. To construct moments for $U$, we define the expected value of an arbitrary random operator $Y$ taking values in $\mathcal{S}(\mathcal{H}_i, \mathcal{H}_j)$ in the usual way, i.e., as any Hilbert-Schmidt operator $\mathrm{E}(Y) \in \mathcal{S}(\mathcal{H}_i, \mathcal{H}_j)$ satisfying
\begin{align*}
    \langle A, \mathrm{E}(Y) \rangle_{ij} = \mathrm{E} \langle A, Y \rangle_{ij},
\end{align*}
for all $A \in \mathcal{S}(\mathcal{H}_i, \mathcal{H}_j)$. By the Riesz representation theorem \cite{conway1990course}, the expectation $\mathrm{E}(Y)$ exists and is unique as soon as $\mathrm{E} \| Y \|_{ij} < \infty$. Defined like this, the expectation of a random operator is straightforwardly verified to satisfy the following intuitive properties (where we implicitly assume that the relevant expectations exist as unique): (i) The expectation is linear in the sense that $\mathrm{E}(a_1 Y_1 + a_2 Y_2) = a_1 \mathrm{E}(Y_1) + a_2 \mathrm{E}(Y_2)$ for all scalars $a_1, a_2 \in \mathbb{R}$ and all random operators $Y_1, Y_2$. (ii) For any random operator $Y$, we have $\mathrm{E}(Y^*) = \{ \mathrm{E}(Y) \}^*$. In particular, if the operator $Y$ is self-adjoint, then so is $E(Y)$. (iii) For any fixed operator $A$ and any random operator $Y$, we have $\mathrm{E}(AY) = A \mathrm{E}(Y)$.

Under Assumption \ref{assu:bounded_on_unit_sphere}, the Hilbert-Schmidt norm of $U$ has a particularly simple upper bound:
\begin{align}\label{eq:operator_norm_bound}
\begin{split}
    \| U \|_{21} \leq& \sum_{j = 1}^r \sigma_j \|  \kappa_1( \cdot,  u_j )\otimes \kappa_2( \cdot,  v_j ) \|_{21} \\
    \leq& (C_1 C_2)^{1/2} r \| X \|_2,
\end{split}
\end{align}
where $ \| X \|_2 = \sigma_1 $ is the operator norm of the random matrix $X$. Consequently, the expectation $\mathrm{E}(U)$ is well-defined and unique as soon as $\mathrm{E} \| X \|_2 < \infty$. However, we instead make the following, stricter assumption that is needed when we next construct the second moment of $U$.

\begin{assumption}\label{assu:second_moment}
We have $\mathrm{E} \| X \|^2_2 < \infty$.
\end{assumption}


By the sub-multiplicativity of the Hilbert-Schmidt norm, $\|  U U^* \|_{11} \leq \| U \|_{21}^2$, showing that Assumption \ref{assu:second_moment} indeed guarantees that $\mathrm{E}(UU^*)$ is well-defined. Having constructed $\mathrm{E}(U)$ and $\mathrm{E}(UU^*)$, we take the operator analogy of the matrix $\mathrm{E} [ \{ X - \mathrm{E}(X) \} \{ X - \mathrm{E}(X) \}' ]$ to be the Hilbert-Schmidt operator $H_1 \equiv H_1(X) \in \mathcal{S}( \mathcal{H}_1, \mathcal{H}_1) $ defined as
\begin{align}\label{eq:operator_h1x}
    H_1 :=& \mathrm{E}[ \{ U - \mathrm{E}(U) \} \{ U - \mathrm{E}(U) \}^* ]\\
    =& \mathrm{E}(UU^*) - \mathrm{E}( U ) \mathrm{E}(U)^*.
\end{align}

In (2D)$^2${PCA}, the left projection directions are obtained as the eigenvectors of the matrix $\mathrm{E} [ \{ X - \mathrm{E}(X) \} \{ X - \mathrm{E}(X) \}' ]$. We next show that the equivalent is well-defined in the non-linear case. Denoting $Y := U - \mathrm{E}(U) $, as $ Y Y^* $ is self-adjoint, so is the operator $H_1$. Moreover, like its linear counterpart, $H_1$ is also positive semi-definite, as is seen by writing, for an arbitrary $f \in \mathcal{H}_1$,
\begin{align*}
    \langle f, H_1 f \rangle_1 =& \langle (f \otimes f), H_1 \rangle_{11} \\
    =& \mathrm{E} \langle (f \otimes f), Y Y^* \rangle_{11}\\
    =& \mathrm{E} \| Y^* f \|_{11}^2.
\end{align*}
As Hilbert-Schmidt operators are compact \cite[p.267]{conway1990course}, $H_1$ admits the spectral decomposition $H_1 = \sum_{k = 1}^\infty \lambda_k (a_k \otimes a_k)$ where $\lambda_1 \geq \lambda_2 \geq \ldots \geq 0$ and $\{ a_k \}$ is an orthonormal basis of $\mathcal{H}_1$ \cite[Theorem 4.10.4]{debnath1999introduction}. We can similarly obtain the orthonormal basis $\{ b_k \}$ of $\mathcal{H}_2$ corresponding to the operator $H_2 := \mathrm{E}[ \{ U - \mathrm{E}(U) \}^* \{ U - \mathrm{E}(U) \} ]$.

We are now in position to define MNPCA. We assume, for now, that the reduced ranks $d_1, d_2$ are known and postpone the discussion of their estimation later to Section \ref{sec:tuning} on tuning parameter selection. Let then $a_1, \ldots, a_{d_1}$ and $b_1, \ldots, b_{d_2}$ be any first $d_1$ and $d_2$ eigenvectors of the self-adjoint positive semi-definite operators $H_1$ and $H_2$, respectively. The $d_1 \times d_2$ matrix $Z$ of non-linear, two-dimensional principal components of $X$ is now found element-wise as
\begin{align}\label{eq:non_linear_extension}
    z_{k\ell} := \langle a_k, \{ U - \mathrm{E}(U) \} b_\ell \rangle_1.
\end{align}
Each row of the matrix $Z$ shares the same non-linear row direction $a_k$ (and analogously for the columns of $Z$), implying that it is meaningful to view $Z$ as a matrix, instead of simply as a collection of latent variables.  Moreover, as lower indices $k$ correspond to larger eigenvalues and greater amount of information, we expect the most interesting part of the matrix $Z$ to be its top left corner. 


\section{Sample consistency}\label{sec:asymptotics}

We next turn our attention to the sample version of MNPCA and its asymptotic properties. Without loss of generality we restrict our discussion to the left-hand side of the model, the equivalent results for the right-hand side following instantly by symmetry.

Let $X_1, \ldots , X_n$ be a sample of $p_1 \times p_2$ matrices from the distribution of $X$ and let $(u_{ij}, v_{ij})$ and $\sigma_{ij}$ denote the $j$th singular pair of the $i$th observed matrix and the corresponding singular value, respectively. For each $i = 1, \ldots , n$, we let $U_i \in \mathcal{B}(\mathcal{H}_2, \mathcal{H}_1)$ denote the linear operator
\begin{align*}
    U_i := \sum_{j = 1}^r \sigma_j \{ \kappa_1( \cdot,  u_j ) \otimes \kappa_2( \cdot,  v_j ) \}.
\end{align*}
Under Assumption \ref{assu:singular_values_simple} and for odd/even kernels $\kappa_1, \kappa_2$, the operators $U_1, \ldots , U_n$ are almost surely unique and a computation similar to \eqref{eq:operator_norm_bound} reveals that they are Hilbert-Schmidt operators. By defining the  ``average'' operator as $\bar{U}_n := (1/n) \sum_{i = 1}^n U_i$, the sample version $H_{n1} \in \mathcal{S}( \mathcal{H}_1, \mathcal{H}_1)$ of the operator $H_1$ is defined as,
\begin{align*}
     H_{n1} := \frac{1}{n} \sum_{i = 1}^n (U_i - \bar{U}_n)(U_i - \bar{U}_n)^*.
\end{align*}
As our first asymptotic result, we show that $H_{n1}$ converges to $H_1$ in the Hilbert-Schmidt norm at the standard root-$n$ rate, as soon as the fourth moment of $X$ is bounded.

\begin{assumption}\label{assu:fourth_moment}
We have $\mathrm{E} \| X \|^4_2 < \infty$.
\end{assumption}

\begin{theorem}\label{theo:hs_convergence_h1}
Under Assumptions \ref{assu:kernel_assumptions}, \ref{assu:singular_values_simple}, \ref{assu:bounded_on_unit_sphere} and \ref{assu:fourth_moment}, we have, as $n \rightarrow \infty$,
\begin{align*}
    \| H_{n1} - H_1 \|_{11} = \mathcal{O}_p \left( \frac{1}{\sqrt{n}} \right).
\end{align*}
\end{theorem}

Under suitable regularity conditions, the convergence of $H_{n1}$ in Theorem \ref{theo:hs_convergence_h1} guarantees that also the corresponding eigenspaces are consistent. A standard assumption for this in the kernel dimension reduction literature is that the operator is question has finite rank and its positive eigenvalues are distinct \cite{zwald2005convergence, li2017nonlinear, li2022dimension}. A finite rank essentially ensures that a sample can be used to capture the full information content of the data, whereas having distinct eigenvalues makes certain that the individual directions can be identified. Under these conditions, the consistency of the eigenvectors follows from \cite[Theorem~2]{zwald2005convergence}.

\begin{assumption}\label{assu:distinct_eval}
The operator $H_1$ has finite rank $d_1$ and its positive eigenvalues are distinct.
\end{assumption}

\begin{corollary}\label{cor:eigen}
    Let Assumptions \ref{assu:kernel_assumptions}, \ref{assu:singular_values_simple}, \ref{assu:bounded_on_unit_sphere}, \ref{assu:fourth_moment} and \ref{assu:distinct_eval} hold. Denote by $a_k$ and $a_{nk}$, $k = 1, \ldots, d_1$, any $k$th eigenvectors of $H_1$ and $H_{n1}$, respectively, with their signs chosen such that $\langle a_k, a_{nk} \rangle_1 \geq 0$. Then, as $n \rightarrow \infty$,
    \begin{align*}
        \| a_{nk} - a_k \|_1 = \mathcal{O}_p \left( \frac{1}{\sqrt{n}} \right).
    \end{align*}
\end{corollary}

Finally, the proof of Theorem \ref{theo:hs_convergence_h1} reveals that $\mathrm{E}(U)$ can be estimated root-$n$-consistently by the operator $\bar{U}_n $, guaranteeing together with Corollary \ref{cor:eigen} that the sample non-linear two-dimensional principal components
\begin{align*}
    z_{i, k\ell} := \langle a_{nk}, ( U_i - \bar{U}_n ) b_{n\ell} \rangle_1, \quad i = 1, \ldots, n,
\end{align*} 
are themselves a good approximation to their population counterparts in \eqref{eq:non_linear_extension}.







\section{Coordinate representation}\label{sec:coordinate}

The results of the preceding section were stated on the operator level, and, in order to apply MNPCA in practice, we next develop a finite-dimensional representation of the method. For this, assume that we are given a sample $X_1, \ldots , X_n$ of $p_1 \times p_2$ matrices from the distribution of $X$. As before, we let $(u_{ij}, v_{ij})$ and $\sigma_{ij}$ denote the $j$th singular pair of the $i$th observed matrix and the corresponding singular value, respectively. 



In standard kernel methodology for vector-valued data, it is typical to take the sample counterpart of the RKHS induced by the kernel $\kappa$ to be the linear span of the evaluation elements $\kappa( \cdot, x_i)$ of the observed sample of vectors $x_1, \ldots , x_n$. In our case of matrix data, the natural counterpart to this procedure is to use the evaluation elements of the singular vectors instead. A key question is then how many singular vectors from each $X_i$ should be used. This choice has a direct impact on the computational complexity of MNPCA as using, say, $m$ singular vectors from each observation yields a kernel matrix of the size $mn \times mn$, leading to increased computational burden for larger $m$. On the other hand, a larger $m$ also guarantees a richer function space, making the choice a trade-off (and $m$ essentially a tuning parameter). For notational simplicity, we have in the following restricted ourselves to using only the first singular vectors, $m = 1$, but the formulas could easily be adapted for other $m$ as well. We thus define the sample counterpart of the RKHS $\mathcal{H}_{1}$ as,
\begin{align*}
    \mathcal{H}_{n1} := \mathrm{span}(\mathcal{B}_{n1}),
\end{align*}
where the spanning set $\mathcal{B}_{n1} := \{ \kappa_1( \cdot, u_{i1}) \mid i = 1 , \ldots , n \}$ is taken to satisfy the following condition, which implies, in particular, that $\mathcal{B}_{n1}$ forms a basis for $\mathcal{H}_{n1}$ and that $\mathrm{dim}(\mathcal{H}_{n1}) = n$.

\begin{assumption}\label{assu:pd_kernel_matrix}
The elements of $\mathcal{B}_{n1}$ are linearly independent.
\end{assumption}

For an arbitrary member $f$ of $\mathcal{H}_{n1}$, we define its coordinate $[f] \in \mathbb{R}^n$ to be the vector of its coefficients in the basis $\mathcal{B}_{n1}$. Thus, letting $k_1:\mathbb{R}^p \to \mathbb{R}^n$ denote the function acting as $k_1(x) = (\kappa_1(x, u_{11}), \ldots , \kappa_1(x, u_{n1}))'$, we have
\begin{align}\label{eq:sample_evaluation_property}
    f(x) = [f]' k_1(x),
\end{align}
for every $x \in \mathbb{R}^{p_1}, f \in \mathcal{H}_{n1}$. Let $K_1 \in \mathbb{R}^{n \times n}$ denote the kernel matrix whose $(i, j)$-element is $\langle \kappa_1(\cdot, u_{i1}), \kappa_1(\cdot, u_{j1}) \rangle_1 = \kappa_1(u_{i1}, u_{j1})$. To turn $\mathcal{H}_{n1}$ into a Hilbert space, we equip $\mathcal{H}_{n1}$ with the inner product
\begin{align}\label{eq:sample_inner_product}
    \langle f, g \rangle_{n1} := [f]' K_1 [g],
\end{align}
whose positive-definiteness is guaranteed by Assumption \ref{assu:pd_kernel_matrix}. Note, however, that the resulting space is, strictly speaking, not an RKHS as we take the domain of its elements to be the full space $\mathbb{R}^{p_1}$ instead of the set $\{ u_{11}, \ldots , u_{n1} \}$. This means, in particular, that the reproducing property $f(x) = \langle f, \kappa_1( \cdot , x) \rangle_{n1}$ holds only when $x \in \{ u_{11}, \ldots , u_{n1} \}$ (however, for $x$ not satisfying this, we still have the relation \eqref{eq:sample_evaluation_property}). Finally, we construct the right-hand side counterparts $\mathcal{H}_{n2}, K_2$ similarly, under the analogous assumptions.

Under the above specifications, the sample algorithm for MNPCA (as described in Section \ref{sec:population}) is given in the next theorem.

\begin{theorem}\label{theo:coordinate_version}
Denote $F_i := \sum_{j=1}^r \sigma_{ij} k_1(u_{ij}) k_2(v_{ij})'$, $i = 1, \ldots , n$, and let $a_1, \ldots , a_{d_1}$ and $b_1, \ldots , b_{d_2}$ be any first $d_1$ and $d_2$ eigenvectors of the $n \times n$ matrices
\begin{align}\label{eq:coordinate_matrix_1}
    P_1 := K_1^{-1/2} \left( \frac{1}{n} \sum_{i=1}^n F_i K_2^{-1} F_i' - \bar{F} K_2^{-1} \bar{F}' \right)  K_1^{-1/2},
\end{align}
and
\begin{align*}
    P_2 := K_2^{-1/2} \left( \frac{1}{n} \sum_{i=1}^n F_i' K_1^{-1} F_i - \bar{F}' K_1^{-1} \bar{F} \right)  K_2^{-1/2},
\end{align*}
respectively, where $\bar{F} := (1/n) \sum_{i=1}^n F_i$. Then, the $d_1 \times d_2$ matrix $Z_i$ of the MNPCA-coordinates of the $i$th observation is given by
\begin{align}\label{eq:coordinate_version}
    z_{i, jk} := a_j' K_1^{-1/2} (F_i - \bar{F}) K_2^{-1/2} b_k.
\end{align}
\end{theorem}

Two notes are in order. First, the proof of Theorem \ref{theo:coordinate_version} reveals that to obtain the MNPCA-coordinate matrix of an out-of-sample observation $X_0$, it is sufficient to replace $F_i$ in \eqref{eq:coordinate_version} with the equivalent matrix having the singular vectors/values of $X_0$ in place of those of $X_i$. Secondly, the presence of the inverses $K_1^{-1}, K_2^{-1}$ might make the procedure numerically unstable in practice, and in our later examples we have replaced them with the corresponding regularized inverses $K^\dagger := ( K + \varepsilon \| K \|_2 )^{-1}$ where $\varepsilon = 0.2$. Similarly, one might want to truncate the decompositions $F_i$ after some small number of singular values, say, two or three (effectively assuming that the rank in Assumption \ref{assu:singular_values_simple} is $d = 2$ or $d = 3$). 

The sample procedure described in Theorem \ref{theo:coordinate_version} can be seen as a true non-linear generalization of (2D)$^2${PCA} in the sense that it reverts back to the standard (2D)$^2${PCA} when a linear kernel is used, as long as the sets of leading singular vectors $u_{11}, \ldots , u_{n1}$ and $v_{11}, \ldots , v_{n1}$ span the full spaces $\mathbb{R}^{p_1}$ and $\mathbb{R}^{p_2}$, respectively. This result, formalized in Theorem \ref{theo:linear_kernel} below, is proven in the appendix. Note also that all linear kernels are odd, satisfying our requirement of odd/even kernels.

\begin{theorem}\label{theo:linear_kernel}
Let $\kappa_1, \kappa_2$ be linear kernels and assume that $n \geq \max\{p_1, p_2\}$ and that the matrices $U := (u_{11}, \ldots , u_{n1})'$, $V := (v_{11}, \ldots , v_{n1})'$ have full rank. Then, treating the inverses in Theorem \ref{theo:coordinate_version} as Moore-Penrose generalized inverses, the two-dimensional principal components of the $i$th observation are,
\begin{align*}
    Z_i := A' (X_i - \bar{X}) B,
\end{align*}
where $A, B$ contain, respectively, any first $d_1$ and $d_2$ eigenvectors of the matrices
\begin{align*}
    \frac{1}{n}\sum_{i=1}^n X_i X_i' - \bar{X} \bar{X}' \quad \mbox{and} \quad \frac{1}{n}\sum_{i=1}^n X_i' X_i - \bar{X}' \bar{X}.
\end{align*}
\end{theorem}

Finally, we still briefly remark on the computational complexity of MNPCA, with comparison to linear (2D)$^2${PCA}. For simplicity, we focus only on the required number of matrix multiplication and eigendecomposition operations. In (2D)$^2${PCA}, computing $(1/n) \sum_{i=1}^n X_i X_i' - \bar{X} \bar{X}'$ requires $n + 1$ multiplications of $p_1 \times p_2$ and $p_2 \times p_1$ matrices, giving the complexity $\mathcal{O}(n p_1^2 p_2)$. Assuming that the latent dimensions $d_1, d_2$ are negligible in size compared to the parameters $n, p_1, p_2$, the extraction of the first $d_1$ eigenvector-eigenvalue pairs of a $p_1 \times p_1$ matrix is an $\mathcal{O}(p_1^2)$-operation. Finally, taking into account also the right-hand side of the model, the computation of the latent matrices in (2D)$^2${PCA} has the total complexity of $\mathcal{O}(n p_1^2 p_2 + n p_1 p_2^2)$. For MNPCA, as detailed in Theorem \ref{theo:coordinate_version}, the computation of the matrix \eqref{eq:coordinate_matrix_1} has $\mathcal{O}(n^4)$ complexity. This is the most expensive operation involved, meaning it is also the complexity of the full procedure. Recall that we assumed earlier that only the first singular space is used to estimate $\mathcal{H}_{1}$ and $\mathcal{H}_{2}$, i.e., $m = 1$. In the case of general $m$, it is simple to check that the complexity of MNPCA is $\mathcal{O}(m^3 n^4)$, verifying that $m$ indeed has a major impact on the complexity.

We thus conclude that, as expected from a kernel method, MNPCA is, in general, slower than its linear counterpart, but that the difference vanishes when the dimensions $p_1, p_2$ grow to become comparable with $n$ in magnitude. Note, also, that the previous computations ignore the complexity involved in computing the kernel functions $k_1, k_2$,  

\section{Tuning parameter selection}\label{sec:tuning}

The tuning parameters of MNPCA include the number $m$ of singular spaces used to approximate the spaces $\mathcal{H}_1$ and $\mathcal{H}_2$, the rank $r$ involved in computing the matrices $F_i$ in Theorem~\ref{theo:coordinate_version}, the latent dimensionalities $(d_1, d_2)$ and any additional tuning parameters involved with the kernels $\kappa_1, \kappa_2$. We next give suggestions on how to choose these in practice. 

The number $m$ of singular spaces differs from the other tuning parameters in the sense that increasing $m$ is always better from the viewpoint of estimation accuracy (barring any possible numerical instability), enabling better coverage of the RKHS. However, as discussed in the previous section, the computational burden of MNPCA increases in the third power of $m$ and, hence, our suggestion is to choose as large value of $m$ as is possible within the given computational limits.

The rank $r$ can be seen to control a bias-variance trade-off on the level of individual observations. That is, too small values of $r$ risk discarding some defining features of the observations $X_i$ whereas large $r$ might bring with it noisy singular spaces, distracting from efficient estimation. In an exploratory context, we suggest experimenting with several small values of $r$, say 1--5, whereas, when using MNPCA as a preprocessing step for another method with measurable performance, cross-validation can also be used. 

As in most forms of PCA, the selection of $d_1$ (and, equivalently $d_2$) can be based on a scree plot of the eigenvalues of the matrix $P_1$ ($P_2$) in Theorem \ref{theo:coordinate_version}. The large dimensionality $n$ of the matrix means that standard cut-offs such as $80\%$ explained variance might not be useful due to the long tail of small but non-zero noise eigenvalues. Instead, we suggest using the following heuristic: retain all principal components whose eigenvalues $\lambda_i$ exceed $\bar{\lambda} + 2 \cdot \mathrm{sd}(\lambda_i)$, separately for the two sides of the model. 

Finally, to choose the tuning parameters of the kernels $\kappa_1, \kappa_2$, an obvious choice is to use cross-validation. Alternatively, in absence of any performance criterion, we suggest using a suitable default value. For example, for the even/odd kernel induced by the Gaussian kernel $(x, y) \mapsto \exp \{ \| x - y \|^2/(2\sigma^2) \}$, a natural choice is to use $\sigma^2 = \| G \|/n$ where the matrix $G$ has the inner product $u_{i1}'u_{j1}$ as its $(i, j)$th element. This choice makes $\sigma^2$ comparable in magnitude to the average value of $\| x - y \|^2$ and is additionally invariant to any sign-changes to individual singular vectors $u_{i1}$. This value will be used as a ``baseline'' in our later data examples.

\section{Data examples}\label{sec:examples}

The \texttt{R}-codes for running MNPCA are available on the author's web page, \url{https://users.utu.fi/jomivi/software}.

\subsection{Simulation}

We next evaluate the performance of MNPCA using simulated image data. As competitors, we take (2D)$^2${PCA} (a linear baseline) and its non-linear extension as proposed by Kong et al. in \cite{kong2005generalized}.
Given $x \in [-\pi, \pi]$ and $\alpha \in (-1, 1)$, we let $u(x; \alpha)$ denote the 10-dimensional vector whose $j$th element equals $\cos \{ (1 - \alpha) ( x -\pi + \frac{j - 1}{10} 2 \pi) \}$. We then fix $\alpha \in (-1, 1)$ and generate images representing two groups as follows: for images from Group 1, we first randomly generate $\theta_1, \theta_2, \theta_3, \theta_4$ i.i.d. from $\mathrm{Unif}(-\pi, \pi)$. A $10 \times 10$ image is then constructed as $u(\theta_1; \alpha)u(\theta_2; \alpha)' + u(\theta_3; \alpha)u(\theta_4; \alpha)'$. An image from Group~2 is generated with identical steps but by using $-\alpha$ in place of $\alpha$, meaning that the parameter $\alpha$ controls the distance between the two groups; larger $\alpha$ corresponds to better separated groups. In this study, we consider a total of three values $\alpha = 0.125, 0.100, 0.075$. Samples of images from the two groups generated with $\alpha = 0.125$ are shown in Figure \ref{fig:example_figure_1}. The two groups (top two vs. bottom two rows in Figure \ref{fig:example_figure_1}) are not that easy to discern visually but, as we will later see, the two groups are actually perfectly separable along a non-linear direction.

\begin{figure}
    \centering
    \includegraphics[width=1.000\textwidth]{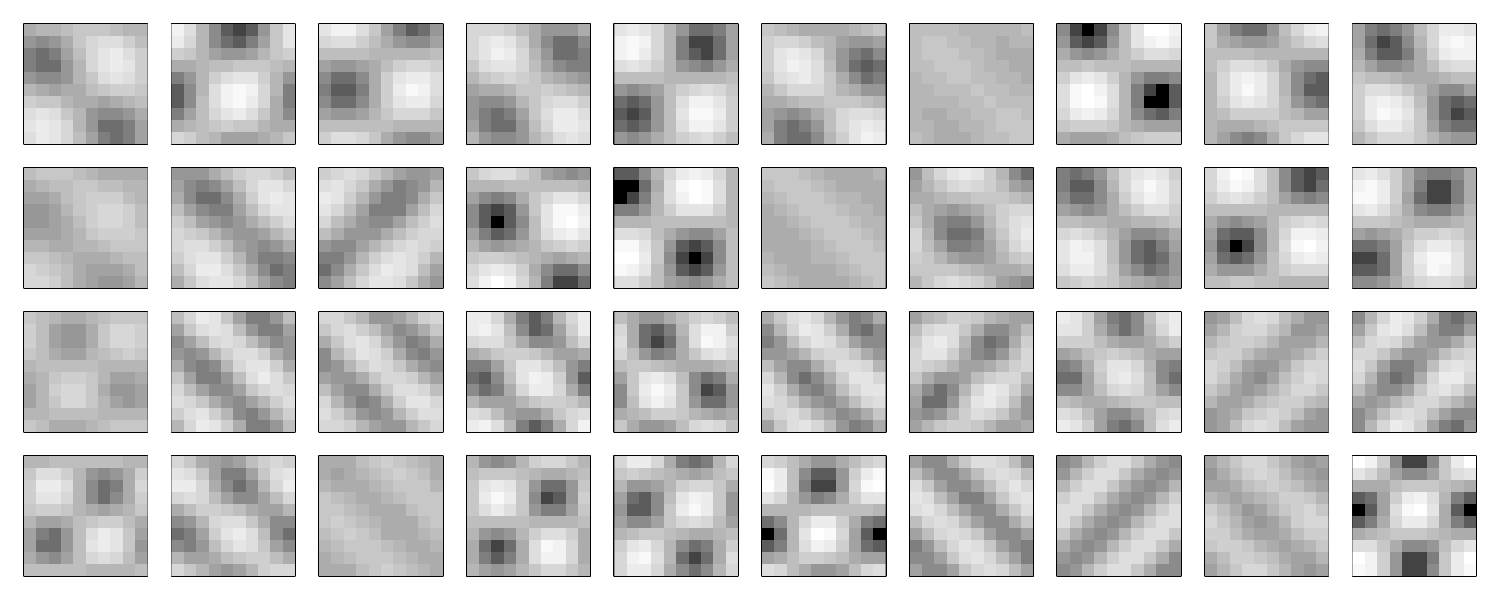}
    \caption{A sample of 20 images from Group 1 (top two rows) and 20 images from Group 2 (bottom two rows). The members of Group 2 can be seen to exhibit a denser checkerboard pattern compared to Group 1.}
    \label{fig:example_figure_1}
\end{figure}

We consider two different sample sizes $n = 50, 100$, generating in both cases 50\% of the observations from each group. In every replicate of the simulation, we use each of the methods to estimate a total of 4 components ($2 \times 2$ latent matrices), fit a QDA classifier to them and, finally, use the trained classifier to predict the classes of a separate test image set of size $n_0 = 50$. We use the Gaussian kernel for all non-linear methods and distinguish two versions of MNPCA, even and odd, giving us a total of four methods to compare. For simplicity and to alleviate computational burden we used $m = 1$ singular spaces to estimate the RKHS and the truncated rank $r = 2$.


The resulting average classification accuracies in the test set over 500 replicates of the simulation are shown in Figure \ref{fig:example_figure_2}. The horizontal axes of the panels correspond to the value of the tuning parameter $\sigma^2$ and are relative in the sense that the tick mark $a$ denotes the value $\sigma^2 = 2^a \sigma^2_0$ where $\sigma^2_0$ is a ``default'' value estimated from the data (once). For MNPCA (the lines denoted by ``Even'' and ``Odd''), we used the default value proposed in Section \ref{sec:tuning} and for the method by \cite{kong2005generalized} (the line ``Kong et al.'') we use the value $\sigma_0^2 = (1/n^2) \sum_{i,j = 1}^n \| x_i - x_j \|^2$ where $x_1, \ldots, x_n$ are the vectorized images. Note that (2D)$^2${PCA} does not use tuning parameters, meaning that its line in Figure \ref{fig:example_figure_2} is perfectly horizontal.

\begin{figure}
    \centering
    \includegraphics[width=1.000\textwidth]{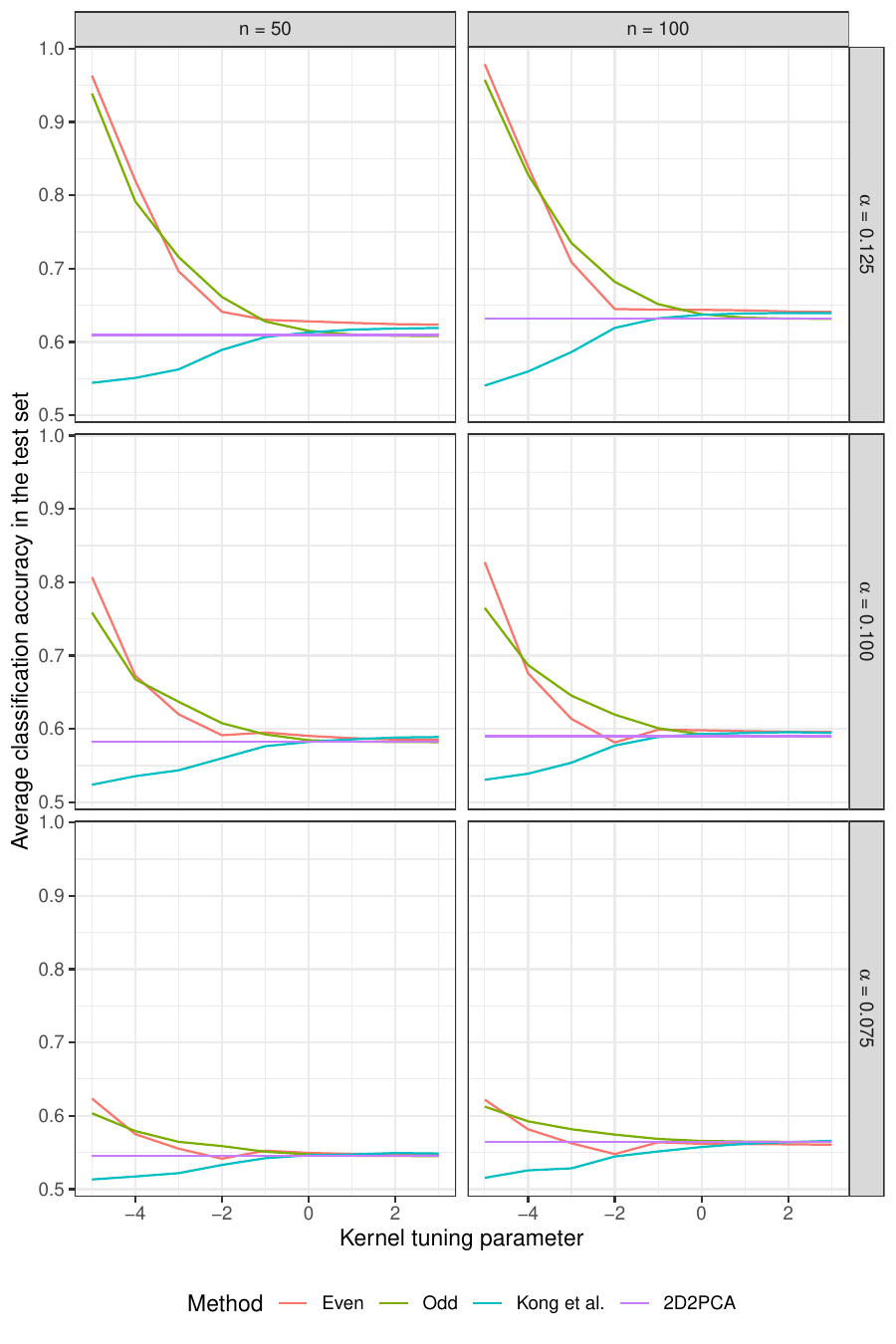}
    \caption{The results of the simulation study. Each line traces the average classification accuracy of the corresponding method as a function of the tuning parameter $\sigma^2$. The panels correspond to the choice of sample size (columns) and the $\alpha$-parameter controlling the difficulty of the separation task (smaller $\alpha$ equals more difficult task).}
    \label{fig:example_figure_2}
\end{figure}

We make the following observations from Figure \ref{fig:example_figure_2}: (i) (2D)$^2${PCA}, while improving over a random guess prediction, fails to reach a satisfactory level of accuracy even in the easiest scenario with $\alpha = 0.125$. This is because the separating direction between the two groups is highly non-linear and, in particular, no single pixel (or even the average behavior of a row or column) can be used to identify whether an image belongs to Group 1 or 2 because of the oscillating mechanism we used to generate the data. (ii) With larger values of $\sigma^2$, Kong et al. manage to improve over (2D)$^2${PCA}. but only by a very small margin. In fact, with an improper choice of the tuning parameter, their accuracy drops well below that of (2D)$^2${PCA}. (iii) Both the even and odd version of MNPCA manage to find the direction (corresponding to small values of $\sigma^2$) that perfectly separates the groups, yielding significantly improved performance over the other two methods. Actually, even though the choice of the tuning parameter $\sigma^2$ greatly affects its performance, MNPCA is still, even for sub-optimally selected $\sigma^2$, roughly as efficient as the competitors at their best. Finally, we also note that there is very little difference between the choice of odd or even Gaussian kernel.



\subsection{Real data example}

We next apply MNPCA to the FashionMNIST data, containing gray scale $28 \times 28$ images of clothing objects and available at Kaggle\footnote{\url{https://www.kaggle.com/datasets/zalando-research/fashionmnist}}. We consider only the 10000 images designated as a ``test set'' and restrict our attention there to the 2000 images of classes 5 and 9, sandals and ankle boots. Figure \ref{fig:mnist_figure_1} illustrates a selection of 40 random images from these two classes. Our objective is the same as in the simulation study, to extract a small amount of components from a training data and use these to fit a QDA classifier for predicting the labels in a separate test data set. We consider two sample sizes $n = 50, 100$ for the training data, taking $n_0 = 50$ test images in both cases. We perform a total of 100 repetitions of the study for both sample sizes, always drawing the training and test sets randomly from the full data set of 2000 images. For the tuning parameters we use the same specifications as in the simulation study.

\begin{figure}
    \centering
    \includegraphics[width=1.000\textwidth]{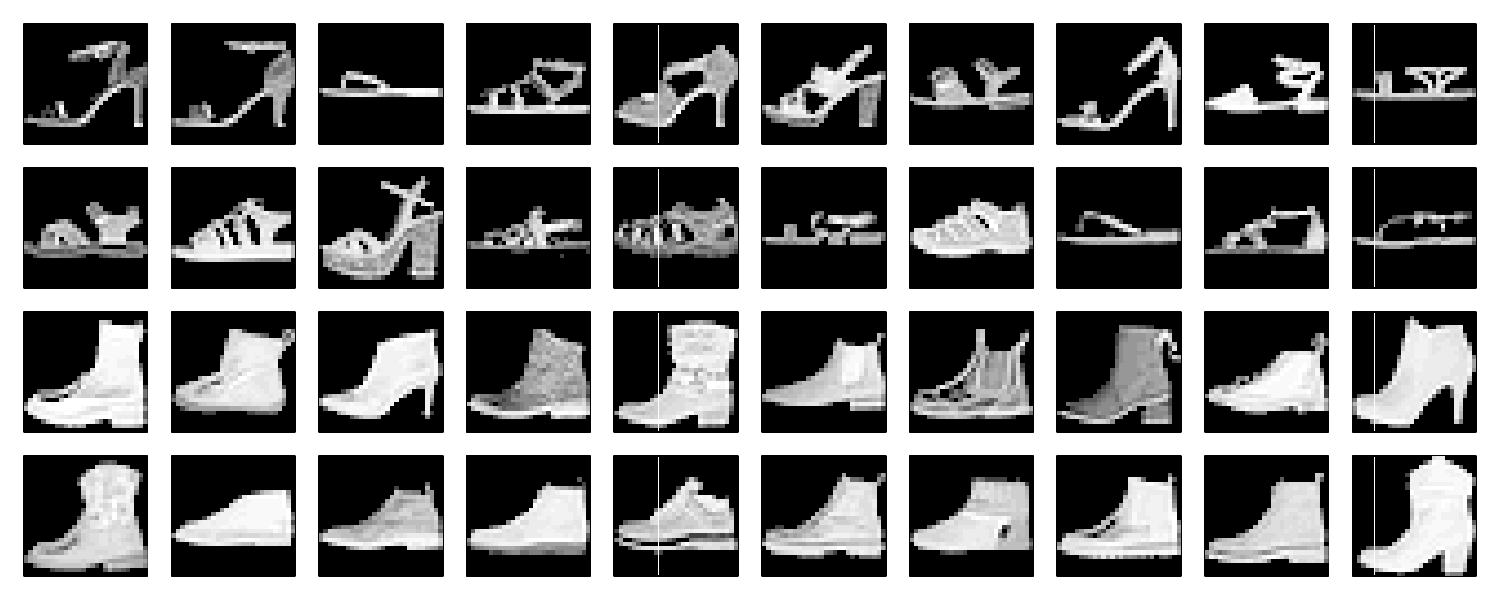}
    \caption{A sample of 20 images from class 5 (sandals, top two rows) and 20 images from class 9 (ankle boots, bottom two rows) in the FashionMNIST data set.}
    \label{fig:mnist_figure_1}
\end{figure}

Due to clear visual differences between the two groups in Figure \ref{fig:mnist_figure_1}, the separating direction is likely to be, if not linear, then at least closely approximable by a linear direction. As such, the differences between the four methods are not expected to be as drastic as in the earlier simulation. Indeed, the results illustrated in Figure \ref{fig:example_figure_2} reveal that this is what happens: (2D)$^2${PCA} is not the best method but it comes very close to the others in terms of classification accuracy. For $n = 50$, the method by Kong et al. achieves, by a small margin, the best performance, whereas, when the sample size is doubled, MNPCA surpasses Kong et al., regardless of the type of the kernel. We also observe that, especially when $n = 100$ and using the odd Gaussian kernel, MNPCA can be seen as a very ``safe'' choice, offering a reliable performance regardless of the value of the tuning parameter $\sigma^2$.

\begin{figure}
    \centering
    \includegraphics[width=1.000\textwidth]{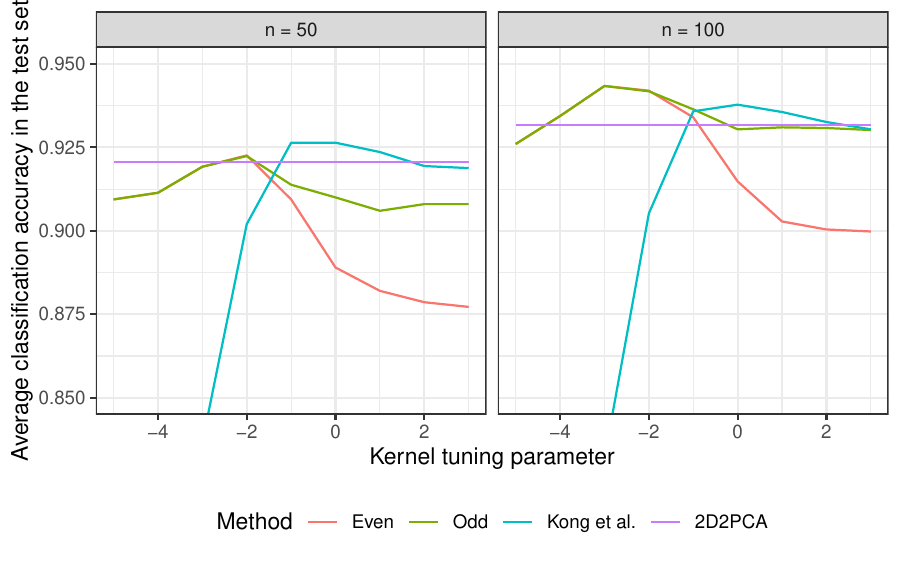}
    \caption{The results of the FashionMNIST data example. See the caption of Figure \ref{fig:example_figure_2} for the interpretation of the plot elements.}
    \label{fig:mnist_figure_2}
\end{figure}

\section{Discussion}\label{sec:discussion}

The work proposed here offers multiple opportunities for future study, which we detail next. Firstly, since its proposal in \cite{zhang20052d}, the linear (2D)$^2${PCA}-paradigm has later been extended to other settings besides PCA, for example, to supervised dimension reduction \cite{li2010dimension} and independent component analysis \cite{virta2017independent}. This naturally begs the question whether the two-sided non-linearization applied in the current work can be extended to these scenarios.

Secondly, in our data examples there was not a major qualitative difference between the odd and even Gaussian kernel. However, this might be context-dependent and it would be interesting to theoretically compare these two classes of kernels. Despite their similarity here, it could be that their behavior differs from one another in some fundamental way.

Thirdly, besides matrix-valued data, also tensor data is currently routinely produced by applications such as medical imaging. As such, extending our non-linearization approach from two-sided to multi-sided would allow for the development of non-linear tensorial dimension reduction methodology. Such an extension is not straightforward as our work here hinges crucially on the use of the singular value decomposition, which is known not to have a direct analogy in the case of tensors \cite{de2000multilinear}.

\section*{Acknowledgments}
The work of JV was supported by the Academy of Finland (grants no. 335077, 347501, 353769).

{\appendix

\section{Proofs of theoretical results}

\begin{proof}[Proof of Theorem \ref{theo:hs_convergence_h1}]

    We begin by establishing that the sample mean operator $\bar{U}_n$ is root-$n$ consistent. Denoting $V_i := U_i - \mathrm{E}(U)$, we have $\mathrm{E}(V_i) = 0$ (the zero operator), for each $i = 1, \ldots , n$. Consequently,
    \begin{align*}
        \mathrm{E} \| \bar{U}_n - \mathrm{E}(U) \|_{21}^2 =& \mathrm{E} \left\langle \frac{1}{n} \sum_{j=1}^n V_j, \frac{1}{n} \sum_{i=1}^n V_i \right\rangle_{21} \\
        =& \frac{1}{n^2} \sum_{i=1}^n \sum_{j=1}^n \mathrm{E} \left\langle V_i, V_j \right\rangle_{21}\\
        =& \frac{1}{n^2} \sum_{i=1}^n \mathrm{E} \| V_i \|_{21}^2\\
        =& \frac{1}{n} \mathrm{E} \| V_1 \|_{21}^2,
    \end{align*}
    where the third equality follows as, for $i \neq j$, we have,
    \begin{align*}
        \mathrm{E} \left\langle V_i, V_j \right\rangle_{21} = \mathrm{E} \{ \mathrm{E} ( \left\langle V_i, V_j \right\rangle_{21} \mid V_j ) \} = \mathrm{E} \left\langle 0, V_j \right\rangle_{21} = 0.
    \end{align*}
    Now, by \eqref{eq:operator_norm_bound}, we have
    \begin{align*}
        \mathrm{E} \| V_1 \|_{21}^2 =& \mathrm{E} \| U_1 \|_{21}^2 - \| \mathrm{E} ( U ) \|_{21}^2\\
        \leq& \mathrm{E} \| U_1 \|_{21}^2\\
        \leq& C_1 C_2 d^2 \mathrm{E} \| X \|_2^2,
    \end{align*}
    where the final upper bound, denoted by $h$ hereafter, is finite under Assumption~\ref{assu:fourth_moment}. Consequently, $ \mathrm{E} \| \bar{U}_n - \mathrm{E}(U) \|_{21}^2 = \mathcal{O}(1/n) $ and Markov's inequality gives, for a fixed $\varepsilon > 0$,
    \begin{align*}
        & \mathrm{P}( \sqrt{n} \| \bar{U}_n - \mathrm{E}(U) \|_{21} \geq \sqrt{h/\varepsilon})\\
        \leq& (\varepsilon/h) n \mathrm{E} \| \bar{U}_n - \mathrm{E}(U) \|_{21}^2\\
        \leq& \varepsilon.
    \end{align*}
    Consequently, $ \| \bar{U}_n - \mathrm{E}(U) \|_{21} = \mathcal{O}_p(1/\sqrt{n})$, as desired.

    Next, an exactly analogous computation shows that, under the given assumptions, we also have $ \| B_n - \mathrm{E}(UU^*) \|_{11} = \mathcal{O}_p(1/\sqrt{n})$ where $B_n := (1/n) \sum_{i = 1}^n U_i U_i'$.

    The consistency of $\bar{U}_n$ further shows that also the operator $\bar{U}_n \bar{U}_n^*$ is consistent:
    \begin{align*}
    & \| \bar{U}_n \bar{U}_n^* - \mathrm{E}(U) \mathrm{E}(U)^* \|_{11} \\
    =& \| \bar{U}_n \{ \bar{U}_n - \mathrm{E}(U) \}^* + \{ \bar{U}_n - \mathrm{E}(U) \} \mathrm{E}(U)^* \|_{11} \\
    \leq& \| \bar{U}_n \{ \bar{U}_n - \mathrm{E}(U) \}^* \|_{11} + \| \{ \bar{U}_n - \mathrm{E}(U) \} \mathrm{E}(U)^* \|_{11} \\
    \leq& \| \bar{U}_n \|_{\mathrm{OP}} \| \bar{U}_n - \mathrm{E}(U) \|_{21} + \| \mathrm{E}(U) \|_{\mathrm{OP}} \| \bar{U}_n - \mathrm{E}(U) \|_{21} \\
    \leq& \| \bar{U}_n \|_{21} \| \bar{U}_n - \mathrm{E}(U) \|_{21} + \| \mathrm{E}(U) \|_{21} \| \bar{U}_n - \mathrm{E}(U) \|_{21} \\
    =& \mathcal{O}_p(1/\sqrt{n}),
    \end{align*}
    where we have used  $\| \mathrm{E}(U) \|_{21} < \infty$ and the following:
\begin{align*}
    \| \bar{U}_n \|_{21} \leq \| \bar{U}_n - \mathrm{E}(U) \|_{21} + \| \mathrm{E}(U) \|_{21} = \mathcal{O}_p(1).
\end{align*}
    The consistency of $H_{n1}$ then follows from the triangle inequality.

\end{proof}

\begin{proof}[Proof of Theorem \ref{theo:coordinate_version}]

Given a linear operator $A$ from one $n$-dimensional Hilbert space to another, we define its coordinate $[A]$ to be the unique $n \times n$ matrix for which $[Af] = [A] [f]$ for all $f \in \mathcal{H}_{n1}$, see, e.g., \cite{lee2013general} for a similar construct. Note that, to avoid notational overload, our ``coordinate notation'' $[\cdot]$ leaves it implicit which basis each instance of the notation refers to (however, this is always obvious from the arguments). A straightforward exercise also reveals that $[AB] = [A][B]$ for any two compatible linear operators $A, B$. Let then the linear operator $U_i \in \mathcal{B}(\mathcal{H}_{n2}, \mathcal{H}_{n1})$ be the $i$th sample counterpart of the random operator $U(X)$ defined in \eqref{eq:reduced_variable}. That is, $U_i$ satisfies,
\begin{align*}
    \langle f, U_i g \rangle_{n1} = \sum_{j=1}^r \sigma_{ij} f(u_{ij}) g(v_{ij}),
\end{align*}
for all $f \in \mathcal{H}_{n1}$ and $g \in \mathcal{H}_{n2}$. Applying our earlier formulas \eqref{eq:sample_evaluation_property} and \eqref{eq:sample_inner_product} then gives that
\begin{align*}
    [f]' K_1 [U_i] [g] = [f]' \left\{ \sum_{j=1}^r \sigma_{ij} k_1( u_{ij} ) k_2( v_{ij} )'  \right\} [g] = [f]' F_i [g],
\end{align*}
valid for all $[f] \in \mathbb{R}^{n}$, $[g] \in \mathbb{R}^{n}$, where $F_i := \sum_{j=1}^r \sigma_{ij} k_1(u_{ij}) k_2(v_{ij})'$. Consequently, we obtain the coordinate
\begin{align*}
    [U_i] = K_1^{-1} F_i,
\end{align*}
using which we obtain the sample counterpart $\bar{U}$ of $\mathrm{E}\{ U(X) \}$ to have the coordinate
\begin{align*}
    [\bar{U}] = \left[ \frac{1}{n} \sum_{i=1}^n U_i \right] = \frac{1}{n} \sum_{i=1}^n [U_i] = K_1^{-1} \bar{F},
\end{align*}
where $\bar{F} := (1/n) \sum_{i=1}^n F_i$. Now, by definition, we have $\langle f, U_i g \rangle_{n1} = \langle U_i^* f, g \rangle_{n2}$ for all $f \in \mathcal{H}_{n1}$ and $g \in \mathcal{H}_{n2}$. Applying \eqref{eq:sample_inner_product} thus reveals that the coordinate of the adjoint of $U_i$ is $[U_i^*] = K_2^{-1} [U_i]' K_1 = K_2^{-1} F_i'$. This in turn shows that the coordinate of the sample counterpart $H_{n1}$ of the (2D)$^2${PCA}-operator $H_1(X)$ in \eqref{eq:operator_h1x} is
\begin{align*}
    [H_{n1}] = K_1^{-1} \frac{1}{n} \sum_{i=1}^n F_i K_2^{-1} F_i' - K_1^{-1} \bar{F} K_2^{-1} \bar{F}'.
\end{align*}
The set of first $d_1$ mutually orthogonal unit-length eigenfunctions $f_1, \ldots , f_{d_1}$ of the operator $H_{n1}$ are thus found as (any) maximizers of the objective function $f \mapsto \langle f, H_{n1} f \rangle_{n1}$ under the constraints that $\langle f_j, f_k \rangle_{n1} = \delta_{jk}$ for all $j = 1, \ldots , d_1$ and $k = 1, \ldots j$, where $\delta_{jk}$ is Kronecker's delta. As $\langle f_j, f_k \rangle_{n1} = [f_j]' K_1 [f_k]$, the vectors $a_j := K_1^{1/2} [f_j]$ thus form an orthonormal set of first $d_1$ eigenvectors of the matrix
\begin{align*}
    K_1^{-1/2} \left( \frac{1}{n} \sum_{i=1}^n F_i K_2^{-1} F_i' - \bar{F} K_2^{-1} \bar{F}' \right)  K_1^{-1/2}.
\end{align*}
This allows us to obtain the coordinates of the desired eigenvectors as $[f_j] = K_1^{-1/2} a_j$ and an analogous derivation for the right-hand side gives us the coordinates $[g_1], \ldots , [g_{d_2}]$ of the right-hand side eigenvectors. The proof is now concluded after observing that the two-sided projection of $U_i - \bar{U}$ to a pair $(f, g)$ is given by $\langle f, (U_i - \bar{U}) g \rangle_{n1} = [f]' (F_i - \bar{F}) [g] $.

\end{proof}

\begin{proof}[Proof of Theorem \ref{theo:linear_kernel}]
    We consider only the left-hand side of the model, the right one behaving analogously. Under linear kernels, we have $F_i = U X_i V'$ (further multiplied by an additional constant if the kernels are scaled). As $U, V$ have full rank, we further have $V' (VV')^\dagger V = I_{p_2}$ where $(VV')^\dagger$ is the Moore-Penrose generalized inverse of the matrix $VV'$, and $ (UU')^{\dagger/2} U = RT'$ where $R \in \mathbb{R}^{n \times p_1}$ and $T \in \mathbb{R}^{p_1 \times p_1}$ contain, respectively, any first $p_1$ left and right singular vectors of $U$. Consequently, the matrix \eqref{eq:coordinate_matrix_1} in Theorem~\ref{theo:coordinate_version} takes the form,
    \begin{align*}
        B_1 := RT' \left( \frac{1}{n}\sum_{i=1}^n X_i X_i' - \bar{X}  \bar{X}' \right) TR'.
    \end{align*}
    As $RT'$ has orthonormal columns, first $d_1$ eigenvectors of $B_1$ are obtained as $RT'A$ where the $p_1 \times d_1$ matrix $A$ contains any orthonormal set of first $d_1$ eigenvectors of the matrix $(1/n) \sum_{i=1}^n X_i X_i' - \bar{X}  \bar{X}'$. Accordingly, when computing the final matrix $Z_i$ of principal components, the centered observation $X_i - \bar{X}$ gets subjected from the left to the linear transformation by the matrix,
    \begin{align*}
        (RT'A)' (UU')^{\dagger/2} U = A' T R' R T' = A',
    \end{align*}
    as claimed, concluding the proof.
\end{proof}

 }

\bibliographystyle{ieeetr}
\bibliography{references}

\begin{thebibliography}{10}

\bibitem{zhang20052d}
D.~Zhang and Z.-H. Zhou, ``{(2D)}$^2${PCA}: Two-directional two-dimensional
  {PCA} for efficient face representation and recognition,'' {\em
  Neurocomputing}, vol.~69, no.~1-3, pp.~224--231, 2005.

\bibitem{gao2016deep}
T.~Gao, X.~Li, Y.~Chai, and Y.~Tang, ``Deep learning with stock indicators and
  two-dimensional principal component analysis for closing price prediction
  system,'' in {\em 2016 7th IEEE international conference on software
  engineering and service science (ICSESS)}, pp.~166--169, IEEE, 2016.

\bibitem{yang2004two}
J.~Yang, D.~Zhang, A.~F. Frangi, and J.-y. Yang, ``Two-dimensional {PCA}: a new
  approach to appearance-based face representation and recognition,'' {\em IEEE
  Transactions on Pattern Analysis and Machine Intelligence}, vol.~26, no.~1,
  pp.~131--137, 2004.

\bibitem{kong2005generalized}
H.~Kong, L.~Wang, E.~K. Teoh, X.~Li, J.-G. Wang, and R.~Venkateswarlu,
  ``Generalized {2D} principal component analysis for face image representation
  and recognition,'' {\em Neural Networks}, vol.~18, no.~5-6, pp.~585--594,
  2005.

\bibitem{nhat2007kernel}
V.~D.~M. Nhat and S.~Lee, ``Kernel-based {2DPCA} for face recognition,'' in
  {\em 2007 IEEE International Symposium on Signal Processing and Information
  Technology}, pp.~35--39, IEEE, 2007.

\bibitem{zhang2006recognizing}
D.~Zhang, S.~Chen, and Z.-H. Zhou, ``Recognizing face or object from a single
  image: Linear vs. kernel methods on {2D} patterns,'' in {\em Joint IAPR
  International Workshops on Statistical Techniques in Pattern Recognition
  (SPR) and Structural and Syntactic Pattern Recognition (SSPR)}, pp.~889--897,
  Springer, 2006.

\bibitem{yu2009k2dpca}
C.~Yu, H.~Qing, and L.~Zhang, ``{K2DPCA} plus {2DPCA}: an efficient approach
  for appearance based object recognition,'' in {\em 2009 3rd International
  Conference on Bioinformatics and Biomedical Engineering}, pp.~1--4, IEEE,
  2009.

\bibitem{fukumizu2004dimensionality}
K.~Fukumizu, F.~R. Bach, and M.~I. Jordan, ``Dimensionality reduction for
  supervised learning with reproducing kernel {H}ilbert spaces,'' {\em Journal
  of Machine Learning Research}, vol.~5, no.~Jan, pp.~73--99, 2004.

\bibitem{liu2013tensorial}
C.~Liu, X.~Wei-sheng, and W.~Qi-di, ``Tensorial kernel principal component
  analysis for action recognition,'' {\em Mathematical Problems in
  Engineering}, vol.~2013, 2013.

\bibitem{shawe2004kernel}
J.~Shawe-Taylor and N.~Cristianini, {\em Kernel Methods for Pattern Analysis}.
\newblock Cambridge University Press, 2004.

\bibitem{hein2004kernels}
M.~Hein and O.~Bousquet, ``Kernels, associated structures and
  generalizations,'' Tech. Rep. 127, Max Planck Institute for Biological
  Cybernetics, 2004.

\bibitem{krejnik2012reproducing}
M.~Krejnik and A.~Tyutin, ``Reproducing kernel {Hilbert} spaces with odd
  kernels in price prediction,'' {\em IEEE transactions on neural networks and
  learning systems}, vol.~23, no.~10, pp.~1564--1573, 2012.

\bibitem{li2010dimension}
B.~Li, M.~K. Kim, and N.~Altman, ``On dimension folding of matrix-or
  array-valued statistical objects,'' {\em Annals of Statistics},
  pp.~1094--1121, 2010.

\bibitem{xue2014sufficient}
Y.~Xue and X.~Yin, ``Sufficient dimension folding for regression mean
  function,'' {\em Journal of Computational and Graphical Statistics}, vol.~23,
  no.~4, pp.~1028--1043, 2014.

\bibitem{ding2015tensor}
S.~Ding and R.~D. Cook, ``Tensor sliced inverse regression,'' {\em Journal of
  Multivariate Analysis}, vol.~133, pp.~216--231, 2015.

\bibitem{virta2017independent}
J.~Virta, B.~Li, K.~Nordhausen, and H.~Oja, ``Independent component analysis
  for tensor-valued data,'' {\em Journal of Multivariate Analysis}, vol.~162,
  pp.~172--192, 2017.

\bibitem{virta2018jade}
J.~Virta, B.~Li, K.~Nordhausen, and H.~Oja, ``Jade for tensor-valued
  observations,'' {\em Journal of Computational and Graphical Statistics},
  vol.~27, no.~3, pp.~628--637, 2018.

\bibitem{conway1990course}
J.~B. Conway, {\em A course in functional analysis}, vol.~96.
\newblock Springer, 1990.

\bibitem{debnath1999introduction}
L.~Debnath and P.~Mikusinski, {\em Introduction to {H}ilbert Spaces with
  Applications}.
\newblock Academic press, 1999.

\bibitem{zwald2005convergence}
L.~Zwald and G.~Blanchard, ``On the convergence of eigenspaces in kernel
  principal component analysis,'' {\em Advances in Neural Information
  Processing Systems}, vol.~18, 2005.

\bibitem{li2017nonlinear}
B.~Li and J.~Song, ``Nonlinear sufficient dimension reduction for functional
  data,'' {\em Annals of Statistics}, vol.~45, no.~3, 2017.

\bibitem{li2022dimension}
B.~Li and J.~Song, ``Dimension reduction for functional data based on weak
  conditional moments,'' {\em Annals of Statistics}, vol.~50, no.~1,
  pp.~107--128, 2022.

\bibitem{de2000multilinear}
L.~De~Lathauwer, B.~De~Moor, and J.~Vandewalle, ``A multilinear singular value
  decomposition,'' {\em SIAM journal on Matrix Analysis and Applications},
  vol.~21, no.~4, pp.~1253--1278, 2000.

\bibitem{lee2013general}
K.-Y. Lee, B.~Li, and F.~Chiaromonte, ``A general theory for nonlinear
  sufficient dimension reduction: Formulation and estimation,'' {\em The Annals
  of Statistics}, vol.~41, no.~1, pp.~221--249, 2013.

\end{thebibliography}

\end{document}